\documentclass[11pt]{amsart}
\usepackage{geometry}                
\usepackage{graphicx}
\usepackage{amssymb}
\usepackage{epstopdf}
\usepackage{tikz}
\usepackage{float}
\newtheorem{theorem}{Theorem}
\newtheorem{prop}[theorem]{Proposition}
\newtheorem{cor}[theorem]{Corollary}
\newtheorem{lemma}[theorem]{Lemma}
\newtheorem{conj}[theorem]{Conjecture}
\theoremstyle{definition}

\newtheorem{example}[theorem]{Example}

\def\la{{\lambda}}

\def\QQ{{\mathbb Q}}

\def\Ht{{\tilde H}}
\def\ga{{\gamma}}
\def\BC{{\mathbb C}}
\def\BB{{\mathbb B}}

\def\Catp{{\rm Cat}'}
\def\area{{\rm area}}
\def\dinv{{\rm dinv}}
\def\bounce{{\rm bounce}}

\def\Rise{{\rm Rise}}
\def\rise{{\rm rise}}
\def\DD{{\mathcal D}}
\def\Sch{{\bf Sch}}
\def\dia{{\rm diag}}
\def\Dia{{\rm Diag}}
\def\touch{{\rm touch}}

\newdimen\squaresize \squaresize=12pt
\newdimen\thickness \thickness=0.4pt

\def\square#1{\hbox{\vrule width \thickness
     \vbox to \squaresize{\hrule height \thickness\vss
        \hbox to \squaresize{\hss#1\hss}
     \vss\hrule height\thickness}
\unskip\vrule width \thickness}
\kern-\thickness}

\def\vsquare#1{\vbox{\square{$#1$}}\kern-\thickness}

\def\thisbox#1{\kern-.09ex\fbox{#1}}
\def\downbox#1{\lower1.200em\hbox{#1}}

\title[A proof of the Delta conjecture Catalan]{A proof of the $4$-variable Catalan polynomial of the Delta conjecture}
\author{Mike Zabrocki}

\begin{document}

\begin{abstract}
In {\it The Delta Conjecture} \cite{HRW15}, Haglund, Remmel and Wilson
introduced a four variable $q,t,z,w$-Catalan polynomial,
so named because the specialization of this polynomial at the values $(q,t,z,w) = (1,1,0,0)$
is equal to the Catalan number $\frac{1}{n+1}\binom{2n}{n}$.  We prove the
compositional version of this conjecture (which implies the non-compositional version)
that states that the coefficient of $s_{r,1^{n-r}}$
in the expression $\Delta_{h_\ell} \nabla C_\alpha$ is equal to a weighted sum over decorated Dyck paths.
\\\\
{\it This paper is dedicated to the memory of Jeffery Remmel (1948--2017).}
\end{abstract}

\maketitle

\begin{section}{Introduction}

In the search for a representation theoretical interpretation for
Macdonald symmetric functions, Haiman defined the module of
diagonal harmonics \cite{Hai94} as a quotient of the polynomial ring
in two sets of $n$ variables.  For a given integer $n$, the diagonal harmonics are
a bi-graded $S_n$-module with dimension $(n+1)^{n-1}$.
Garsia and Haiman \cite{GarHai96} took a
(at the time conjectured) formula for the bi-graded Frobenius characteristic of the diagonal harmonics
and defined for each $n$ a rational function in two parameters
$q$ and $t$ which is equal to the bi-graded multiplicity of
the alternating representation in the module.
This expression is known
as the $q,t$-Catalan polynomial \cite{GarHai96} since at $q=t=1$ it specializes to the Catalan
number $\frac{1}{n+1}\binom{2n}{n}$.  In 2000, Garsia and Haglund \cite{Hag03, GarHag02} 
announced a proof
that the $q,t$-Catalan was a polynomial 
in $q$ and $t$ with non-negative integer coefficients and provided a combinatorial interpretation
for the expression in terms of Dyck paths. 

Important progress was made in the development of that proof through the
introduction of the notation of two linear symmetric function
operators $\nabla$ and $\Delta_f$
that have Macdonald symmetric functions as eigen-functions \cite{BG99, BGHT99}.
The expression $\nabla( e_n )$ was conjectured to be
equal to the Frobenius image of the character of the module of diagonal
harmonics and the $q,t$-Catalan polynomial is the coefficient of $e_n$ in this
expression.  The operators $\nabla$ and $\Delta_f$ gave notation to
extend the types of symmetric function expressions which were conjectured
to be Schur positive for representation theoretic reasons to expressions
which are conjectured to be Schur positive because of computer experimentation.

Haglund conjectured \cite{Hag03} and shortly after Garsia and
Haglund \cite{GarHag02} proved a combinatorial interpretation for the
$q,t$-Catalan polynomial.  They showed that there were two statistics on Dyck paths
(called $\area$ and $\bounce$) such that the rational expression for the
$q,t$-Catalan is equal to the sum over all Dyck paths $D$ with weight
$q^{\area(D)}t^{\bounce(D)}$.  Around this same period, Haiman \cite{Hai02} proved
the conjecture that $\nabla(e_n)$
was equal to the Frobenius
image of the graded character of the diagonal harmonics.  Haiman also guessed at a second statistic
($\dinv$, short for {\it diagonal inversions}) such that the $q,t$-Catalan is equal to
the sum over all Dyck paths with weight $q^{\dinv(D)}t^{\area(D)}$ and
Haglund later showed with a bijection why the two combinatorial expressions are equivalent.

With the conjectures on the $q,t$-Catalan polynomial resolved,
Haglund, Haiman, Loehr, Remmel and Ulyanov \cite{HHLRU05} extended the
combinatorial interpretations for the coefficient
of $e_n$ in $\nabla(e_n)$ to
other coefficients.
They conjectured the coefficient of any monomial symmetric function 
in terms of labelled Dyck paths (also known as parking functions) and this
became known as {\it the Shuffle Conjecture}.
The Shuffle Conjecture takes its name because the coefficient of
a monomial is equal to the number of labelled Dyck paths whose reading
word is a shuffle of segments of length the parts of the partition.

Researchers also considered coefficients of $\nabla$ and $\Delta_f$ acting
on other symmetric functions and extended the combinatorial interpretations
to coefficients in these expressions (e.g. \cite{EHKK03, Hag04, CL06, LW07, LW08}
and for a survey of results in this area up to 2008 see \cite{Hag08}).

In particular, a refinement of the Shuffle Conjecture was proposed by Haglund, Morse and
the author \cite{HMZ12} that gave a symmetric function expression for the
labelled Dyck paths which touch the diagonal at a given composition.  Some progress
on this {\it Compositional Shuffle Conjecture} was made \cite{GXZ10, Hic10, DGZ13, 
Hic14, GXZ14a, GXZ14b} before it was finally
proven by Carlsson and Mellit \cite{CM15}.
By the time that Carlsson and Mellit had announced their proof, there was already a
rational slope version of the compositional shuffle conjecture \cite{BGLX16}.
The arms race of conjecture vs. proof in this area did not stay out of balance for long and
a proof of this result was announced in 2016 by Mellit \cite{Mel16}.

Haglund, Remmel and Wilson \cite{HRW15} recently announced a conjecture for some
combinatorial expressions involving $\Delta_f$ and $\nabla$ in a
sequence of conjectures that generalize the Shuffle Conjecture
from labelled Dyck paths to decorated labelled Dyck paths
and called this {\it the Delta Conjecture}.
There does not currently exist a compositional version of this
conjecture which might be helpful if progress is to be made on proving it.

They noticed however that the coefficients of a Schur function indexed by
a hook in the expression $\Delta_{h_\ell} \nabla e_{n}$
had similar behavior to the $q,t$-Catalan \cite{GarHai96, GarHag02} and $q,t$-Schr\"oder \cite{EHKK03, Hag04}
and they proposed a four parameter expression $C_n(q,t,w,z)$
and a combinatorial interpretation for this expression in terms of decorated Dyck paths.
In fact they proposed two combinatorial interpretations and one of them
is compatible with the compositional refinement proposed by Haglund, Morse and the author \cite{HMZ12}.
It is this conjecture that we shall prove here.

Although this is not precisely how the combinatorial interpretation was formulated in \cite{HRW15},
we will present it here in terms of decorated Schr\"oder paths.  Schr\"oder paths were used as a
combinatorial description for the coefficients of Schur function indexed by a hook in the expression
$\nabla(e_n)$ in \cite{EHKK03,Hag04}.  In this paper we will give a combinatorial description for
the coefficient of a Schur function indexed by a hook in the expression $\Delta_{h_\ell} \nabla( e_{n-\ell})$
as Schr\"oder paths with $\ell$ vertical segments decorated with a $\circ$ symbol.

In fact, a Schr\"oder path is simply a Dyck path with some of the peaks in the Dyck path
changed to $NE$-diagonal steps.  In all of the Schr\"oder paths
we will also insist that the right most peak
in the highest diagonal not have a $NE$-diagonal step.\footnote{In an early version
of \cite{HRW15}, the combinatorial
interpretation was stated in terms of
$\circ$-decorated Dyck paths where there is a difference
in the decorations on the peaks and double rises.  The latest version does not
express the combinatorial interpretation for the coefficients in terms of
decorated Dyck paths at all, but it is a useful construction in relating the
right and left hand side of Theorem \ref{thm:mainresult}.
Here we use Schr\"oder paths to distinguish the
$\circ$-decorations on the peaks (which are diagonal edges here)
from those on the double rises.}

A Schr\"oder path is a generalization of a Dyck path that is a lattice path in the $n \times n$ square
that start in the South-West corner and go to the North-East corner allowing for steps North,
East and diagonal steps which are North-East such that the path
stays above the South-West/North-East diagonal.
A {\it $\circ$-decorated Schr\"oder path} is a Schr\"oder path in which some of the vertical steps
which are not peaks are decorated with a $\circ$.  We will show that
the coefficient of $s_{k+1,1^{n-k-\ell-1}}$ in $\Delta_{h_\ell} \nabla e_{n-\ell}$ is a
$q,t$ enumeration of the $\circ$-decorated Schr\"oder paths of length $n$ with $k$ diagonal North-East steps
and $\ell$ vertical segments decorated with $\circ$.  In particular we will show that 
$\langle \Delta_{h_\ell} \nabla( e_{n-\ell}), e_{n-\ell} \rangle$
is a positive polynomial in $q,t$ that enumerates $\circ$-decorated Dyck paths.

\begin{figure}[h]
\begin{center}
\includegraphics[width=2.5in]{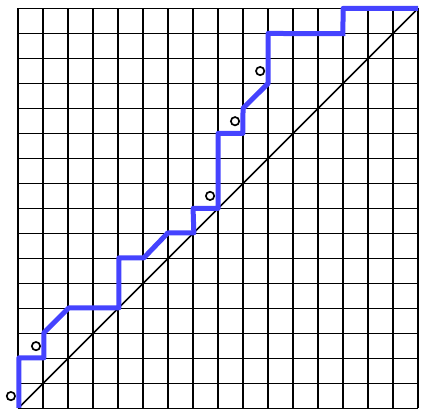}
\end{center}
\caption{An example of a $\circ$-decorated Schr\"oder path.
In this example, $\rise_\circ(D)=5$ because there are five 
$\circ$-decorated vertical segments 
and $\dia(D)=3$ because there are three NE-diagonal steps.
The usual touch composition is $(4,3,1,8)$, but the
rise-touch composition is equal to $(2,3,1,5)$ because two vertical segments are $\circ$-decorated
in the first touch segment and three vertical segments are $\circ$-decorated in the last touch
segment. \label{fig:sample}}
\end{figure}

For a given $\circ$-decorated Schr\"oder path $P$,
the number of $\circ$-decorations on the path will be denoted $\rise_\circ(P)$
and the number of diagonal NE steps will be denoted $\dia(P)$.
The positions where the Schr\"oder path touches the diagonal divides the path
into segments and determines a composition,
$\touch(P) = (\alpha_1, \alpha_2, \ldots, \alpha_{\ell(\alpha)})$ where $\alpha_i$ is the
length of the $i^{th}$ segment.   We will also consider the rise-touch
composition $\touch_\circ(P) = (\beta_1, \beta_2, \ldots, \beta_{\ell(\alpha)})$
where $\beta_i$ is equal to $\alpha_i$ minus the number of $\circ$-decorations in
the $i^{th}$ segment.  There are two additional statistics on these
paths $\area_\circ(P)$ and $\dinv_\circ(P)$ which we will explain in detail in Section
\ref{sec:comb}.

Namely we will show the following theorem
(this is Theorem \ref{thm:mainresult}; note: we leave the definition of
the statistics $\dinv_\circ$ and $\area_\circ$
to Section \ref{sec:comb}):
\begin{theorem}  For non-negative integers $n, k, \ell$ and a composition $\alpha$ of size $n-\ell$,
\begin{equation} 
\left< \Delta_{h_\ell} \nabla C_\alpha, s_{k+1,1^{|\alpha|-k-1}} \right> =
\sum_{\substack{P \in \Sch^\circ_{n}\\ \dia(P)=k, \rise_\circ(P) = \ell\\
\touch_\circ(P) = \alpha}}
q^{\dinv_\circ(P)} t^{\area_\circ(P)}
\end{equation}
where the sum is over all $\circ$-decorated Schr\"oder paths with $k$ NE-diagonal steps
and $\ell$ $\circ$-decorated rises and rise-touch composition equal to $\alpha$.
\end{theorem}

The symmetric function $C_\alpha$ which appears in this theorem is compositional form of
a Hall-Littlewood symmetric function which was introduced in \cite{HMZ12}.  The definition
appears in section \ref{sec:compsf}.  By Proposition 5.2 of \cite{HMZ12},
$e_n = \sum_{\alpha \models n} C_\alpha$ and hence we have given a combinatorial interpretation
for $\langle \Delta_{h_\ell} \nabla e_{n-\ell}, s_{k+1,1^{n-\ell-k-1}} \rangle$.

Note that the resolution of this conjecture does not prove all of the conjectures
made in Section 7 of \cite{HRW15} because there was a second combinatorial interpretation stated
for the coefficient $\left< \Delta_{h_\ell} \nabla e_{n-\ell}, s_{k+1,1^{n-k-1}} \right>$
that does not seem to be compatible with the compositional version and our techniques
depend strongly on compatibility with the compositional construction.
\end{section}

\begin{section}{Symmetric functions} \label{sec:sf}

The results related to Macdonald symmetric functions that we will use here almost all come
from a series of early papers on the subject
\cite{Mac88,G92,GarHai95,GarHai96,BG99, BGHT99,GHT99,GarHag02}.  These results have
proven to be very prescient in the utility of the
identities, notation and techniques developed.  We will be able to prove
our symmetric function recurrence
by using the groundwork paved in these references.  The book by J. Haglund \cite{Hag08} collects
many of the identities that we will use in a review of the literature and hence will provide a useful
reference for their use.  The only additional ingredient that
we will use are the creation operators and symmetric functions introduced in \cite{HMZ12}
which play an important role in developing recurrences for the coefficients in which we
are interested.

\begin{subsection}{Symmetric function notation}
The main reference we will use for symmetric functions is \cite{Mac95}.  The standard bases
of the symmetric functions that will appear in our calculations the complete $\{ h_\lambda \}_\lambda$, 
elementary $\{ e_\lambda \}_\lambda$, power $\{ p_\lambda \}_\lambda$ and Schur $\{ s_\lambda \}_\lambda$ bases.

The ring of symmetric functions can be thought of as the polynomial ring in the
power sum generators $p_1, p_2, p_3, \ldots$.  
As we are working with Macdonald symmetric functions involving
two parameters $q$ and $t$ we will consider this polynomial ring over the field $\QQ(q,t)$.

We will make extensive use of plethystic notation in our calculations and arbitrary alphabets.  This is
a notational addition that introduces union and difference of alphabets.  Alphabets will be
represented as sums of monomials $X = x_1 + x_2 + x_3 \ldots$ and then the
expression $f[X]$ represents the symmetric function $f$ as an element of $\Lambda$ with $p_k$
replaced by $x_1^k + x_2^k + x_3^k + \cdots$.
We have the identities that $p_k[X+Y] = p_k[X] + p_k[Y]$, $p_k[X-Y] = p_k[X] - p_k[Y]$, and on
the elementary and homogeneous bases we also have the alphabet addition formulae which say
\begin{equation}\label{eq:alphaaddition}
e_n[X+Y] = \sum_{k=0}^n e_k[X] e_{n-k}[Y]\hbox{ and }h_n[X+Y] = \sum_{k=0}^n h_k[X] h_{n-k}[Y]~.
\end{equation}
The notation $\epsilon$ is a common tool to express a second sort of negative sign when working with
symmetric functions with alphabets
where $p_k[\epsilon X] = (-1)^k p_k[X]$.  This is different from the negative
of the alphabet expressed as $p_k[-X] = - p_k[X]$.  In general
$f[- \epsilon X] = (\omega f)[X]$ where $\omega$ is the fundamental algebraic involution which sends
$e_k$ to $h_k$, $s_\lambda$ to $s_{\lambda'}$ and $p_k$ to $(-1)^{k-1} p_k$.

There is a special element $\Omega$ in the completion of the symmetric functions that we will be using.
It is defined as $\Omega = \sum_{n \geq 0} h_n$.  It has the property for arbitrary alphabets $X$ and $Y$,
$\Omega[X + Y] = \Omega[X]\Omega[Y]$~.
In addition, it has the property that for any two dual basis $\{ a_\lambda \}_\lambda$ 
and $\{b_\lambda\}_\lambda$ with respect to the standard scalar product $\left< s_\lambda, s_\mu \right> = \chi(\lambda=\mu)$,
we have
\begin{equation}
\Omega[XY] = \sum_{\lambda} a_\lambda[X] b_\lambda[Y]~.
\end{equation}
\end{subsection}

\begin{subsection}{Macdonald symmetric function toolkit and $q,t$ notation}
Macdonald symmetric functions that are used here are a transformation of the bases presented in \cite{Mac95}.
They are the symmetric functions that are the Frobenius image of the Garsia-Haiman modules \cite{GarHai93}
indexed by a partition.
The symmetric functions
\begin{equation}
\Ht_\mu[X;q,t] = \sum_{\lambda \vdash |\mu|} K_{\lambda\mu}(q,1/t) t^{n(\mu)} s_\lambda[X]
\end{equation}
where $K_{\lambda\mu}(q,t)$ are the Macdonald $q,t$-Kostka coefficients and $n(\mu) = \sum_{i \geq 1} (i-1) \mu_i$.
The basis elements $\{ \Ht_\mu \}_\mu$
are orthogonal with respect to the scalar
product
\begin{equation}\label{eq:defstarscalar}
\left< p_\lambda, p_\mu \right>_\ast = \chi(\lambda=\mu) 
(-1)^{|\la|+\ell(\la)}\prod_{i=1}^{\ell(\la)}(1-q^{\lambda_i})(1-t^{\lambda_i})
\end{equation}
and are sometimes defined by this property.

\begin{figure}[h]
\begin{center}
\includegraphics[width=2in]{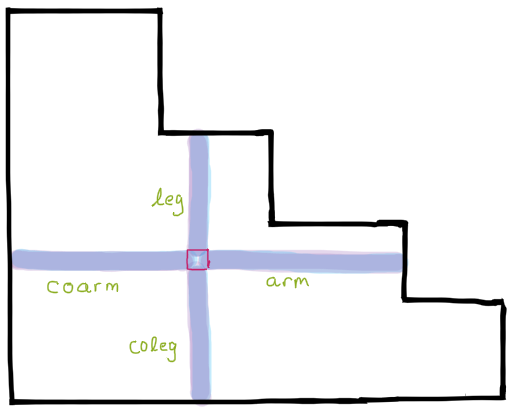}
\end{center}
\begin{caption}{arm, leg, co-arm and co-leg of a cell of the diagram}\label{fig:armleg}
\end{caption}
\end{figure}

If we identify the partition $\mu$ with the collection of cells $\{ (i,j) : 1 \leq i \leq \mu_i, 1 \leq j \leq \ell(\mu) \}$,
then for each cell $c \in \mu$ we refer to the the arm, leg, co-arm and co-leg
(denoted respectively as $a_\mu(c), \ell_\mu(c), a'_\mu(c), \ell_\mu'(c)$)
as the number of cells in the segments labeled in Figure \ref{fig:armleg}.  The typical shorthand
for the polynomial expressions in $q$ and $t$ are
\begin{equation*}
B_\mu = \sum_{c \in \mu} q^{a_\mu'(c)} t^{\ell_\mu'(c)}, 
T_\mu = \prod_{c \in \mu} q^{a_\mu'(c)} t^{\ell_\mu'(c)}
\hbox{ and }w_\mu = \prod_{c \in \mu} (q^{a_\mu(c)}-t^{\ell_\mu(c)+1})(t^{\ell_\mu(c)}-q^{a_\mu(c)+1})~.
\end{equation*}
Also set $M = (1-q)(1-t)$ and $D_\mu = M B_\mu -1$.

The following linear operators were introduced in \cite{BG99, BGHT99} which are at the basis of the
conjectures relating symmetric function coefficients and $q,t$-combinatorics in this area.  Define
\begin{equation}
\nabla(\Ht_\mu) = T_\mu \Ht_\mu \hbox{ and } \Delta_{f}( \Ht_\mu ) = f[B_\mu] \Ht_\mu~.
\end{equation}
Note that if $n = |\mu|$, then $e_n[ B_\mu ] = T_\mu$, hence for a symmetric function $f$ of
homogeneous degree $n$, $\Delta_{e_n}(f) = \nabla(f)$, so the operators $\Delta_f$ are seen
as a more general operator than $\nabla$.

Following other references and introduce the shorthand notation
$f^\ast[X] = f \!\!\left[ \frac{X}{M} \right]~.$
This notation can then be used to relate the $\ast$-scalar product with the usual scalar product
$\left< f, g \right>$ where the Schur functions are orthonormal
since $\left< f, g \right> = \left< f, \omega g^\ast \right>_\ast$.  It is known that
$\left< \Ht_\lambda, \Ht_\mu\right>_\ast = \chi(\lambda=\mu) w_\lambda$, then
follows that
\begin{equation} \label{eq:enexpansion}
\Omega\left[\frac{- \epsilon XY}{M}\right] = \sum_{n \geq 0} e_n^\ast[XY] =
\sum_{\mu\vdash n} \frac{\Ht_\mu[X] \Ht_\mu[Y]}{w_\mu}~.
\end{equation}

We will use one of the forms of Macdonald-Koornwinder reciprocity in our calculations
 (see \cite{Mac95} p. 332 or \cite{GHT99}),
\begin{equation}
\frac{\Ht_\mu[1+u D_\lambda]}{\prod_{c \in \mu} (1- u t^{\ell'(c)} q^{a'(c)})}
= 
\frac{\Ht_\la[1+u D_\mu]}{\prod_{c \in \la} (1- u t^{\ell'(c)} q^{a'(c)})}~.
\end{equation}
The form of this identity that we are most interested here is found by 
setting $u = 1/u$, clearing the denominators, and letting $u \rightarrow 0$
to obtain
\begin{equation} \label{eq:GHTreciprocity}
\Ht_\mu[D_\nu] = (-1)^{|\mu|+|\nu|} \Ht_\nu[D_\mu] \frac{T_\mu}{T_\nu}~.
\end{equation}
\end{subsection}
\begin{subsection}{Pieri rules and summation formulae}
Define coefficients $h_1^\perp \Ht_\mu = \sum_{\nu \rightarrow \mu} c_{\mu\nu} \Ht_\nu$ and
$h_1 \Ht_\nu = \sum_{\mu \leftarrow \nu} d_{\mu\nu} \Ht_\mu$.
It was proven in \cite{GarHai95} (Corollary 1.1) that they are related by the identity,
\begin{equation}\label{eq:cmunudmunurel}
d_{\ga\tau} =  M c_{\ga\tau} \frac{w_\tau}{w_\ga}~.
\end{equation}

The following identity has been used frequently in
work on the Shuffle Conjecture
but a full proof did not appear until recently in \cite{GHXZ16}.
For $s \geq 0$,
\begin{equation}\label{eq:esDga}
e_{s-1}\big[D_\gamma \big] =
(-1)^{s-1} \sum_{\nu \leftarrow \gamma} d_{\nu\gamma} \left(\frac{T_\nu}{T_\gamma}\right)^{s} + \chi(s=0)
\end{equation}
where we have denoted $\chi(true) = 1$ and $\chi(false) = 0$ so that the term $\chi(s=0)$ only appears
in the case that $s=0$.  The other sum of Pieri coefficients for Macdonald polynomials
was proven in a $q,t$ hook walk by Garsia and Haiman \cite{GarHai95} for $s \geq 0$,
\begin{equation}\label{eq:hkDga}
h_{s+1}[D_\ga] = 
M t^{s} q^{s} \sum_{\tau \rightarrow \ga} c_{\ga\tau} \left( \frac{T_\ga}{T_\tau} \right)^{s} - \chi(s=0)~.
\end{equation}

In order to prove the combinatorial formula for the $q,t$-Catalan polynomial,
Garsia and Haglund introduced a generalization of
the Pieri coefficients and proved a summation formula which we will use here.
They defined coefficients $d_{\mu\nu}^{f}$ and $c_{\mu\nu}^{f\perp}$ where $\nu \subseteq \mu$ as
\begin{equation}\label{eq:genpiericoeffs}
f \Ht_\nu = \sum_{\mu} d_{\mu\nu}^f \Ht_\mu \hskip .3in\hbox{ and }\hskip .3in
f^\perp \Ht_\mu = \sum_{\nu} c_{\mu\nu}^{f\perp} \Ht_\nu~.
\end{equation}
These coefficients are related by
\begin{equation}\label{eq:genPiericoefrel}
c_{\mu\nu}^{f\perp} w_\nu = d_{\mu\nu}^{\omega f^\ast} w_\mu~.
\end{equation}
The summation formula from \cite{GarHag02} (see pp. 698-701) we will use here is
\begin{equation}\label{eq:genPierisums}
\sum_{\substack{\nu \subseteq \mu\\m-d\leq |\nu|\leq m}}
c_{\mu\nu}^{g \perp} = \nabla^{-1}\left( (\omega g)\!\! \left[ \frac{X- \epsilon}{M} \right] \right) 
\Big|_{X \rightarrow D_\mu}
\end{equation}
where $\mu \vdash m$ and  $g$ is a symmetric function of degree less than or equal to $d$.
\end{subsection}
\begin{subsection}{Symmetric functions indexed by compositions and creation operators}\label{sec:compsf}
The work of Haglund, Morse and the author \cite{HMZ12} extended the Shuffle Conjecture
to a compositional refinement.
The Compositional Shuffle Conjecture implies the original Shuffle Conjecture, and it was this
version of the conjecture that was proven in \cite{CM15}.

The compositional refinement came by defining for each composition
$\alpha$ symmetric functions $B_\alpha[X;q]$
and $C_\alpha[X;q]$.  These symmetric functions have the property that
the combinatorial expression in terms of labeled Dyck paths 
for $\nabla B_\alpha[X;q]$ is in terms of paths which touch in {\it at least} in the positions
specified by the composition $\alpha$ and $\nabla C_\alpha[X;q]$ which touches the
diagonal in {\it exactly} the positions specified by the composition $\alpha$.

Both of these symmetric functions are defined in terms of creation operators.
For any
symmetric function $P[X]$ define
\begin{align}\label{eq:bopdef}
\BB_m P[X] = P\!\!\left[X + \epsilon \frac{(1-q)}{u}\right]\Omega[-\epsilon uX]\Big|_{u^m}
\end{align}
and
\begin{align}\label{eq:copdef}
\BC_m P[X] 
= (-q) P\!\!\left[X + \epsilon\frac{(1-q)}{u}\right]\Omega\!\!\left[\frac{\epsilon uX}{q}\right]\Big|_{u^m}~.
\end{align}
Then for any composition $\alpha = (\alpha_1, \alpha_2, \ldots, \alpha_{\ell(\alpha)})$, set
$C_\alpha = C_\alpha[X;q] = \BC_{\alpha_1} \BC_{\alpha_2} \ldots \BC_{\alpha_{\ell(\alpha)}}(1)$.
We can define the symmetric functions $B_\alpha$ in a similar manner, but for our
purposes, we only need (\cite{HMZ12} equation (5.11) and (5.12)) the $C_\alpha$ symmetric functions and
the fact that for $m \geq 0$, 
\begin{equation}\label{eq:BCrelation}
\BB_m (C_\alpha) = q^{\ell(\alpha)} \sum_{\beta \models m} C_{\alpha,\beta}~.
\end{equation}

We will denote $\BB_m^\ast$ and $\BC_m^\ast$ operators which are dual
to $\BB_m$ and $\BC_m$ with respect to the $\ast$-scalar product
(that is $\left< \BB_m f, g \right>_\ast = \left< f, \BB_m^\ast g \right>_\ast$).
Since $\left< e_n^\ast[XY], f[X] \right>_\ast =f[Y]$, 
then for any symmetric function $f[X]$,
$$\left< \BB_m^X e_n^\ast[XY], f[Y] \right>_\ast = \BB_m f[X] = 
\left<e_n^\ast[XY], \BB_m^{Y} f[Y] \right>_\ast =
\left<\BB_m^{\ast Y} e_n^\ast[XY], f[Y] \right>_\ast$$
where we have denoted $\BB_m^X$ to be the $\BB_m$ operator acting on the $X$ variables and
$\BB_m^{\ast Y}$ to be the $\ast$-dual operator acting on the $Y$ variables.  We conclude
that the following two expressions 
be verified by calculating that $\BB_m^X e_n^\ast[XY] = \BB_m^{\ast Y} e_n^\ast[XY]$
(and similarly $\BC_m^X e_n^\ast[XY] = \BC_m^{\ast Y} e_n^\ast[XY]$).
These operators are expressed by the formulae, 
\begin{equation}\label{eq:Bopsaction}
\BB_m^* P[X] = P\!\!\left[X+\frac{M}{u}  \right]\Omega\!\!\left[\frac{-uX}{1-t} \right]
\Big|_{u^{-m}}
\end{equation}
and
\begin{align}
\BC_m^* P[X]
&= (-q)
P\!\!\left[X-\frac{ M}{qu}  \right]
\Omega\!\!\left[\frac{- u X}{1-t} \right]
\Big|_{u^{-m}}~.\label{eq:Copsaction}
\end{align}
\end{subsection}
\end{section}

\begin{section}{The combinatorial recurrence} \label{sec:comb}
In \cite{HRW15} there are two combinatorial interpretations stated in Conjecture 7.1 for
the symmetric function expression $\left< \Delta_{h_\ell} \nabla e_{n-\ell}, s_{k+1,1^{n-\ell-k-1}} \right>$
and only one of these two seems to be compatible with the
coefficient in the expression
$\left< \Delta_{h_\ell} \nabla C_\alpha, s_{k+1,1^{n-\ell-k-1}} \right>$
(where $\alpha$ is a composition and $k,\ell \geq 0$)
in terms of decorated Dyck paths.

What we will do in the beginning of this section is to introduce the definitions
necessary to state the combinatorial interpretation. In Section \ref{sec:combrec} we
state and prove
a recurrence on the generating function for the combinatorial
objects.  Then in Section \ref{sec:sfrec} we will show a symmetric
function identity that demonstrates the coefficients also satisfy the same
recurrence.  This will imply by an inductive argument (because for small
values of the indices we can verify that the combinatorial values agree
with the symmetric function coefficients) that the symmetric function coefficients agree
with the combinatorial generating function.

The basic element of the recursive construction given to us by the symmetric function
recurrence is a rotation of the first part around to the end of the Dyck path while
deleting the first up step and the first right step that touches the diagonal.  This
combinatorial recurrence and the effect on the $\area$ and $\dinv$ statistics
first appears in \cite{Hic10}.
Assuming that we know the size of the piece that is being rotated around, the process
is reversible.  This is exactly the same sort of recursive construction that appeared
in the proofs of certain coefficients of the Compositional Shuffle Conjecture \cite{GXZ10, GXZ14a, GXZ14b}.

\begin{figure}[h] 
\begin{center}
\begin{tabular}{p{0.5\textwidth} p{0.5\textwidth}}
\vspace{0pt}\includegraphics[width=2.3in]{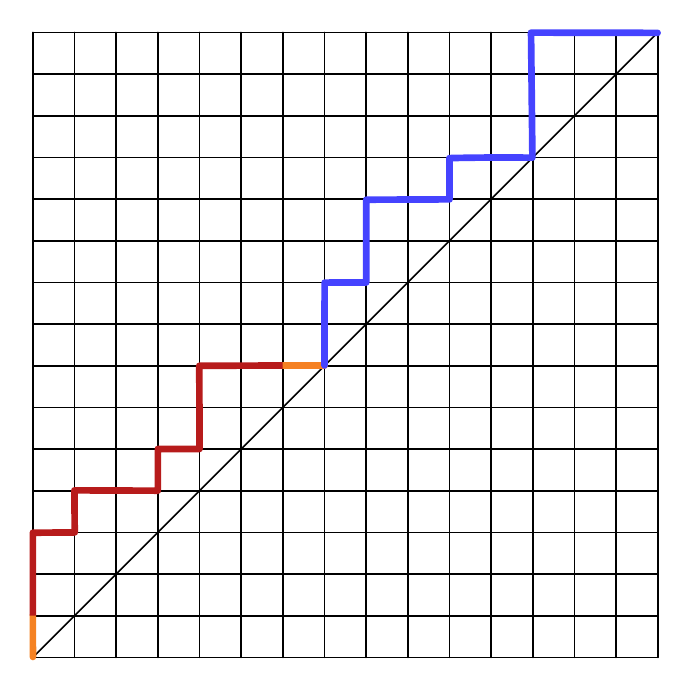}&
\vspace{3pt}\includegraphics[width=2.2in]{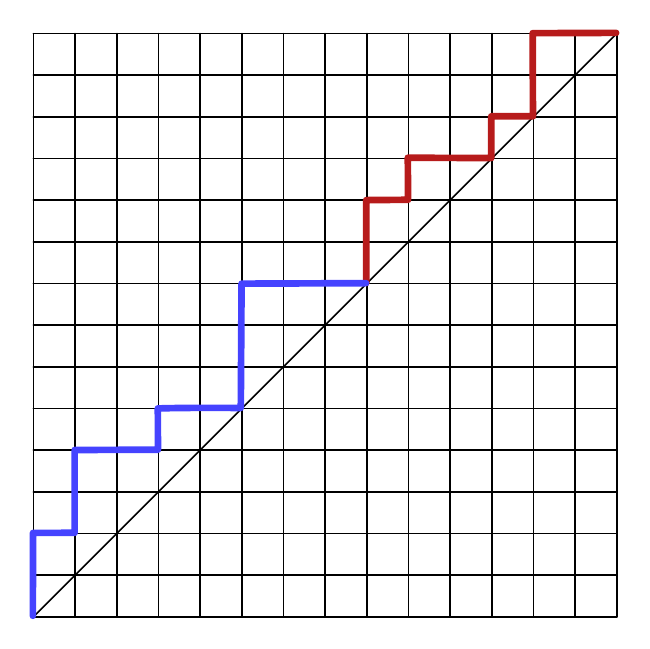}
\end{tabular}
\end{center}
\begin{caption}{Example of cyclic rotation} \label{fig:exampleDP}
The Dyck path on the left first touches the diagonal after 7 vertical/horizontal steps and has $\area(D) = 17$.
The Dyck path on the right has the red part of the path moved to the end with the orange segments removed.  The resulting
path has area $11 = 17 - (7-1).$
\end{caption}
\end{figure}

Vertical edges in a Dyck path
come in two types, a {\it peak} is a vertical edge followed 
by a horizontal edge, a {\it double rise} is a
a vertical edge followed by a second vertical edge.
A Schr\"oder path is a Dyck path with some of the peaks
changed to $NE$-diagonal edges.  The Schr\"oder paths we
will work with will have the restriction that the rightmost peak in the
highest diagonal cannot be a $NE$-diagonal edge.  A decorated Schr\"oder path
will be a Schr\"oder path with some of the vertical edges which are not
peaks decorated with a $\circ$.
Since a $\circ$-decorated Schr\"oder path is just a Dyck path with some of the vertical edges
either decorated or diagonal, we will explain the recurrence on the Dyck path
and assume that the decorations travel with the edges in the recurrence.
In this way we will identify a decorated Schr\"oder path with its underlying Dyck path and we
will partly do this by calling the path $P$ and the underlying Dyck path $D$.

Dyck paths may be encoded by the area sequence $(a_1(D), a_2(D), \ldots, a_n(D))$ where 
$a_i(D)$ is the number of full cells in the the $i^{th}$ row which are above the diagonal but below the
Dyck path.  The area sequence of the
Dyck paths are characterized by the property that $a_1(D)=0$ and
$0 \leq a_{i+1}(D) \leq a_{i}(D) +1$ for $1 \leq i < n$.
The area statistic on Dyck paths is $\area(D) = \sum_{i=1}^n
a_i(D)$.

The diagonal inversion statistic on Dyck paths is the number of pairs $(i,j)$ with $i<j$ such that
either $a_i(D) = a_j(D)$ or $a_i(D)=a_j(D)+1$ (the diagonal inversions of the Dyck path).
Order the
rows from largest area value to smallest and from right to left.  The $i^{th}$ row in this order
has area $a_k(D)$ for some $k$ and let $b_i(D)$ be the
number of diagonal inversions of the form $(k,j)$ with $a_k(D)=a_j(D)$ or $(j,k)$ with $a_j(D) = a_k(D)+1$.
The $b_i(D)$ represents the number of diagonal inversions between the $i^{th}$ vertical step in this
order and all those that come before.
Set $\dinv(D) = \sum_{i=1}^n b_i(D)$.

\begin{example}
To ensure that that the definitions are clear to this point we list the sequences for the Dyck paths
pictured in Figure \ref{fig:exampleDP}.  The $a$-sequence is
$$(0,1,2,2,1,1,2,0,1,1,2,1,0,1,2),$$
and the $b$-sequence is $$(0, 1, 2, 3, 4, 4, 5, 5, 6, 6, 7, 6, 6, 4, 2)~.$$
The $a$ and $b$ sequences of the right Dyck path have a clear relationship to the left Dyck path.
The statistics for the right path have the $a$-sequence  given by
$$(0,1,1,2,1,0,1,2,0,1,1,0,0,1),$$
and the $b$-sequence is $$(0, 1, 2, 3, 4, 4, 5, 5, 6, 6, 7, 6, 6, 4)$$
found by deleting the last entry.
\end{example}

The combinatorial interpretation of the symmetric function coefficients
are in terms of decorated Schr\"oder paths.
The set of $\circ$-decorated Schr\"oder paths are Dyck paths where
each peak except the rightmost one in the highest diagonal can be
replaced with a $NE$-diagonal step and each
vertical segment that is not a peak can be labeled with 
a $\circ$ or not.  Denote the set of
$\circ$-decorated Schr\"oder paths in an $n \times n$ square by $\Sch_n^\circ$
and the set of Dyck paths of the same size by $\DD_n$.  The cardinality of 
$\Sch_n^\circ$ is equal to $\frac{2^{n-1}}{n+1} \binom{2n}{n}$.
This enumeration follows because for each Dyck path in $\DD_n$ there are $2^{n-1}$ corresponding 
Schr\"oder paths since each of the $n-1$ vertical edges of the Dyck path
(not the rightmost highest peak)
has two choices as being either labeled/diagonal or not labeled/diagonal.

A {\it $\circ$-decorated rise} on a $\circ$-decorated
Schr\"oder path is a row where the first vertical segment
of two consecutive vertical segments has a $\circ$-decoration
(i.e. a $\circ$-decorated row with $a_{i}(D) = a_{i+1}(D)-1$).
Let $\Rise_\circ(P)$ be the set of indices of
the rows of the $\circ$-decorated rises and
$\rise_\circ(P) = |\Rise_\circ(P)|$ be the number of $\circ$-decorated
rises.
We will also set for $\circ$-decorated Schr\"oder paths,
$\area_\circ(P) = \area(D) - \sum_{i \in \Rise_\circ(P)} a_{i+1}(D) = 
\area(D) - \rise_\circ(D) - \sum_{i \in \Rise_\circ(D)} a_{i}(D)$.

There is a reading order of
the vertical segments of a Dyck path which are ordered by
reading them from the highest diagonal from right to left.
We will denote the set of indices of the diagonal edges in a corresponding Schr\"oder path $P$
(following the reading order) by $\Dia_\circ(P)$ and
the number of $\circ$-decorated peaks by $\dia_\circ(P) = |\Dia_\circ(P)|$.

We note that because we restricted in our $\circ$-decorated Schr\"oder paths
that the rightmost peak in the highest diagonal
cannot be a $NE$-diagonal edge, $1 \notin \Dia_\circ(P)$.
Also remark that rows that form a peak in a Dyck path will have $b_i(D) > b_{i-1}(D)$
(except where $b_1(D)=0$).
A peak of a Dyck path will have a diagonal inversion with all the same positions that the
preceding vertical segment had, plus one with that previous vertical segment.

Denote a restricted diagonal inversion statistic on $\circ$-decorated Schr\"oder paths
$$\dinv_\circ(P) = \dinv(D) - \sum_{i \in \Dia_\circ(P)} b_i(D)~.$$

Define the following multivariate analogues of $\frac{2^{n-1}}{n+1} \binom{2n}{n}$,
\begin{align}
\Catp_n(q,t,z,w) &= \sum_{D \in \DD_n} q^{\dinv(D)} t^{\area(D)} \prod_{b_i(D)>b_{i-1}(D)}
\left( 1 + z/q^{b_i(D)} \right)\\
&\hskip .2in\times \prod_{a_i(D)>a_{i-1}(D)} (1 + w/t^{a_i(D)})\\
&= \sum_{P\in \Sch_n^\circ } q^{\dinv_\circ(P)} t^{\area_\circ(P)} z^{\dia_\circ(P)} w^{\rise_\circ(P)}~.
\end{align}


There is a compositional refinement that we are focusing on in this paper.
For a $\circ$-decorated Dyck path $D$ with $\rise_\circ(D) = \ell$,
define the rise-touch composition of $D$
to be the sequence of numbers of
vertical steps which are not $\circ$-decorated rises
between the places the path touches the diagonal.
This is more simply determined by looking at the
usual touch composition for the Dyck path without the decorations
and then subtracting from each part of the composition the number of $\circ$-decorated rises in each segment.
Where it is necessary to abbreviate, the rise-touch composition will be denoted $\touch(D)$.
It is the case that $\touch(D)$ is a composition of $n - \ell$.

%

\begin{example}
To ensure that the combinatorial object and the definitions 
that we are considering are clear, consider the Schr\"oder path that appears
in Figure \ref{fig:sample}.  The Schr\"oder path $P$ has an underlying Dyck path $D$
and the 
$a$-sequence for this path is $(0,1,1,2,$ $0,1,1,0,$ $0,1,2,2,$ $3,3,4,2)$ and the
 $b$-sequence for this path is $(0, 0, 1, 2, 1, 2, 3, 1, 2, 3, 3, 4, 4, 5, 4, 3)$.
The $area(D) = 23$ and the so the $area_\circ(P) = 23 - (1+2+1+3+4) = 12$.
The $dinv(D) = 38$ and the peaks corresponding to the indices $3, 7, 9$ in the
reading order are $NE$-diagonal edges and so $dinv_\circ(P) = 38 - (1+3+2) = 32$.
\end{example}

We next define a generating function for the $\circ$-decorated
Dyck paths such that the rise-touch composition is equal to $\alpha$.
Fix non-negative integers $k$ and $\ell$ and a composition $\alpha$.
Then set
\begin{equation}
\Catp_{\alpha,k,\ell}(q,t) = \sum_{\substack{P \in \Sch^\circ_{|\alpha|+\ell}\\ \dia_\circ(P)=k\\
\touch_\circ(P) = \alpha}}
q^{\dinv_\circ(P)} t^{\area_\circ(P)}
\end{equation}
By definition,
\begin{equation}
\Catp_n(q,t,z,w) = \sum_{k,\ell: k+\ell \leq n-1} 
\sum_{\alpha \models n-\ell} \Catp_{\alpha,k,\ell}(q,t) z^k w^\ell~.
\end{equation}
this is because if there are $\ell$ rises which are $\circ$-decorated, then the
rise-touch composition of the $\circ$-decorated Dyck path will be some composition $\alpha$ of
$n-\ell$ and so both the left and right hand side of the equation
is equal to a weighted sum over all $\circ$-decorated Dyck paths of size $n$.

%

\begin{example} The following $8$ Schr\"oder paths have $touch_\circ(P) = (2,2)$.
The third path in each row has $\area_\circ(P) = 3$, the others have $\area_\circ(P)=2$.

\begin{center}
\includegraphics[width=4in]{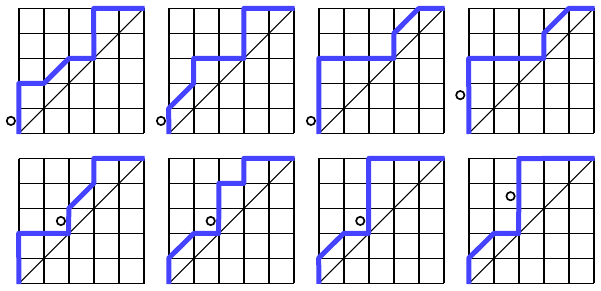}
\end{center}
The first two of these paths have $b$-sequence $(0, 1, 2, 2, 1)$,
the second two have a $b$-sequence $(0, 1, 1, 1, 1)$.  In the bottom row
the first two paths have $b$-sequence $(0, 1, 2, 1, 1)$ and the last two
have $b$-sequence $(0, 0, 1, 1, 1)$.  As a consequence the $\dinv_\circ(P)$
for the paths are $5, 4, 3, 3$ for the top row and $4, 3, 2, 2$ for the bottom
row.

The generating function for these objects is
\begin{equation}
\Catp_{(2,2),1,1}(q,t) = 
q^{5} t^{2} + q^{4} t^{2} + q^{3} t^{3} +q^{3} t^{2} + q^{4} t^{2} + q^{3} t^{2} + q^{2} t^{3}
+ q^{2} t^{2}~.
\end{equation}
\end{example}

\begin{subsection}{Combinatorial recurrence on $\Catp_{\beta,k,\ell}(q,t)$} \label{sec:combrec}

We will show in this subsection that for non-negative integers $k, \ell$ such that
$k < |\alpha|$ that $\Catp_{\alpha,k,\ell}(q,t)$ satisfies a recurrence involving the first
part of the partition.
We consider two cases, one where the first part of the composition $\alpha$ is greater than $1$
and a second where $\alpha_1 = 1$.
\begin{prop} \label{prop:Catrec1} For $a \geq 1$, and $k,\ell \geq 0$, and a fixed composition $\beta$
we have the combinatorial recurrence,
\begin{align}\label{eq:firststepbig}
\Catp_{(a+1,\beta),k,\ell}(q,t)
= &t^{a} q^{\ell(\beta)} \sum_{\ga \models a} \Catp_{(\beta,\ga),k,\ell}(q,t)
+ t^{a} q^{\ell(\beta)} \sum_{\ga \models a+1} \Catp_{(\beta,\ga),k,\ell-1}(q,t)~.
\end{align}
\end{prop}

\begin{proof} To prove equation \eqref{eq:firststepbig}, we will divide the set of
$\circ$-decorated Dyck paths into two subsets: either the first vertical step of the $\circ$-decorated
Dyck path is decorated ({\bf Case 1}) or it is not decorated ({\bf Case 2}).
On both sets we will perform a cyclic rotation
as described at the beginning of this section.  We will refer to the piece of the Dyck path
up to the first point that it touches the diagonal as {\it first} (the part pictured
in red in Figure \ref{fig:exampleDP}) and the piece of the Dyck path after the first part as {\it rest}
(the part pictured in blue in Figure \ref{fig:exampleDP}).

The circle decorations travel with the vertical edges of the path
and the vertical edge which is deleted may either be decorated or not.
The sequence of $b$-values for the cyclically rotated path is exactly the same as the
$b$-sequence for the original path except the last entry (which corresponds to the deleted edge)
is deleted.  This is because
diagonal inversions within the first piece or within the rest piece are still diagonal inversions after
the cyclic rotation and diagonal inversions between the first piece and the rest piece will 
switch from being on the same diagonal to being on the diagonal below (and vice versa).

We also remark that the rightmost peak in the highest diagonal
in the path before rotation will remain the rightmost
peak in the highest diagonal in the path after the operation.
Therefore the condition that the rightmost peak in the highest diagonal is not a NE edge is
a condition which must hold and come from the base case.

\vskip .1in
\noindent {\bf Case 1}: Cyclic rotation gives a bijection between $\circ$-decorated Schr\"oder paths $P$
with rise-touch composition $(a+1,\beta)$ and $\dia_\circ(P)=k$ 
and $\rise_\circ(P)=\ell$ 
where {\bf the first vertical step is not decorated} and 
$\circ$-decorated Schr\"oder paths $P'$
with rise-touch composition of the form $(\beta,\gamma)$ with $\gamma \models a$ 
and $\dia_\circ(P')=k$ 
and $\dia_\circ(P')=\ell$. 

Say that there are $r$ $\circ$-decorated rises in the first piece.
Since the first vertical step is not decorated, after
deleting the first edge from a piece of a Dyck path which touches the diagonal after $a+r+1$ steps,
the rise-touch composition will be a composition $\gamma$ of size $a$ and deleting the first
vertical step has the effect of removing $a$ cells contributing to the area (one for each 
peak or non-decorated rise
row in the first piece of the Dyck path),
so $\area_\circ(D) = \area_\circ(D') + a$.
Moreover since there is one diagonal inversion of the form $(1,j)$ for each $j>1$ with $a_j(D)=0$, 
then $b_n(D) = \ell(\beta)$ (where $n$ is equal to the length of $D$) 
and $\dinv_\circ(D) = \dinv_\circ(D') + \ell(\beta)$.

\begin{example} \label{ex:singletouchunlabeled} Consider the first part of the Dyck path
being the following piece of length 8.  The first part of the rise-touch composition of this
Dyck path will be $8$ minus the number of rises which are $\circ$-decorated in this piece (in this case there
are $2$).
\begin{center}
\includegraphics[width=1.5in]{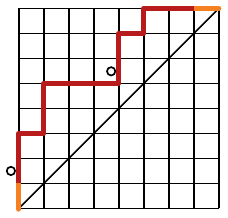}
\end{center}
We consider this example to see how the rise-touch composition and $\circ$-area is affected by deleting the first vertical
step and last horizontal step and ensure that the definitions are clear to this point.
Since the $\circ$-decorations are in rows $2$ and $6$, the $\circ$-area contribution 
from this first piece is $9$ and length of the first entry of the rise-touch
composition is $6$.  Deleting the first vertical and last horizontal step then the rise-touch composition
will have contribition $(3,2)$ and the $\circ$-area statistic will contribute $4 = 9-(6-1)$ from this piece.
\end{example}

\vskip .1in
\noindent {\bf Case 2}: Cyclic rotation also gives a bijection between $\circ$-decorated Schr\"oder paths
$P$ with rise-touch composition $(a+1,\beta)$ and $\dia_\circ(P)=k$ 
and $\rise_\circ(P)=\ell$ 
where {\bf the first vertical step is decorated} and $\circ$-decorated Dyck paths $P'$
with rise-touch composition $(\beta,\gamma)$ with $\gamma \models a+1$ 
and $\dia_\circ(P')=k$ 
and $\rise_\circ(P')=\ell-1$.  Deleting the $\circ$-decoration from the first vertical step will reduce the
number of $\circ$-decorations by one.

Say that there are $r$ $\circ$-decorated rises in the first piece and the Dyck path first touches
the diagonal after $a+1+r$ steps.
Notice that the rise-touch composition of the first piece of the path after deleting the first step
will be a composition $\gamma$ of size $a+1$.  Deleting the first
vertical step has the effect of removing $a$ cells contributing to the area (one for each 
peak or non-decorated rise
row in the first piece of the Dyck path),
so $\area_\circ(D) = \area_\circ(D') + a$.
Moreover since there is one diagonal inversion of the form $(1,j)$ for each $j>1$ with $a_j(D)=0$, 
then $b_n(D) = \ell(\beta)$ and $\dinv_\circ(D) = \dinv_\circ(D') + \ell(\beta)$.

\begin{example}\label{ex:singletouchlabeled}
Consider the same first piece of the Dyck path as in Example \ref{ex:singletouchunlabeled} but assume
now that rows $1, 2$ and $6$ are labelled.  The first entry of the rise-touch
composition will be $5$ which is the same size as the resulting rise-touch composition of $(3,2)$ when
we delete the first vertical and last horizontal step.
\begin{center}
\includegraphics[width=1.5in]{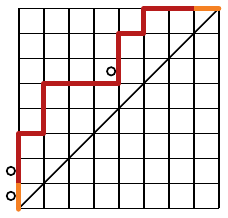}
\end{center}
The $\circ$-area of the contribution from the first piece is $8$ before deleting the first
vertical and last horizontal step and it is $4$ after.
\end{example}

The two cases are disjoint and together cover all $\circ$-decorated Dyck paths
and the weights agree between those on the left
and right hand side, so equation \eqref{eq:firststepbig} holds.
\end{proof}

Next we consider the case where $\alpha = (1,\beta)$ (that is, the $a=0$ case) and we see that the
recurrence is slightly different.

\begin{prop} \label{prop:Catrec2} For $k,\ell \geq 0$, and a composition $\beta$
we have the combinatorial recurrence,
\begin{align}\label{eq:firststepsmall}
\Catp_{(1,\beta),k,\ell}(q,t)
= &q^{\ell(\beta)} \Catp_{\beta,k,\ell}(q,t)+ \Catp_{\beta,k-1,\ell}(q,t)
+ q^{\ell(\beta)} \Catp_{(\beta,1),k,\ell-1}(q,t)~.
\end{align}
\end{prop}

\begin{proof}
In the case that the first part of the rise-touch composition is $1$, there are three types of
$\circ$-decorated Dyck paths which contribute to this expression:
the Dyck path starts with a non-decorated vertical step
followed by a horizontal step ({\bf Case 1}), it begins with a $\circ$-decorated vertical step
followed by a horizontal step ({\bf Case 2}) or it begins with some number of $\circ$-decorated
rises followed by a peak followed by the same number of horizontal steps ({\bf Case 3}) (in the picture
below, the example representing this case has two $\circ$-decorated rises,
but in general there are potentially between $1$ and $\ell$ $\circ$-decorated rises
and the peak can potentially be a $NE$-diagonal step).

\begin{center}
\includegraphics[width=1.5in]{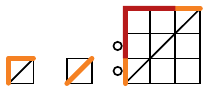}
\end{center}

\vskip .1in
\noindent {\bf Case 1}: The $\circ$-decorated Schr\"oder paths $P$ which begin with a non-decorated vertical step
followed by a horizontal step with rise-touch composition of the form $(1,\beta)$ and
$\dia_\circ(P)=k$
and $\rise_\circ(P)=\ell$ are in bijection with the $\circ$-decorated
Schr\"oder paths $P'$ with rise-touch composition
$\beta$ and $\dia_\circ(P)=k$
and $\rise_\circ(P)=\ell$ by removing the first vertical and horizontal steps.
Since there are $\ell(\beta)$ diagonal inversions of the form $(1,j)$ for each $j>1$ with
$a_j(D)=0$, then $\dinv_\circ(P) = \dinv_\circ(P')+\ell(\beta)$.  The area doesn't change
by deleting the first vertical and horizontal step so $\area_\circ(P) = \area_\circ(P')$.

\vskip .1in
\noindent {\bf Case 2}: The $\circ$-decorated Schr\"oder paths $P$
with rise-touch composition of the form $(1,\beta)$
which begin with a $NE$-diagonal step
and $\dia_\circ(P)=k$
and $\rise_\circ(P)=\ell$ are in bijection with the $\circ$-decorated Schr\"oder paths $P'$
with rise-touch composition
$\beta$, $\dia_\circ(P)=k-1$
and $\rise_\circ(P)=\ell$.  The bijection is simply to remove the first vertical and horizontal steps
(and hence one of the decorations on the peaks).
Since the first step of the path is a $NE$-diagonal, it does not
contribute to the $\circ$-$\dinv$ statistic and $\dinv_\circ(P) = \dinv_\circ(P')$.  Moreover the
$\circ$-area does not change by removing the first vertical and horizontal step
so $\area_\circ(P) = \area_\circ(P')$.

\vskip .1in
\noindent {\bf Case 3}: The $\circ$-decorated Schr\"oder paths
$P$ with rise-touch composition equal to $(1, \beta)$
which begin with a sequence of $\circ$-decorated rises,
followed by a peak or a $NE$-diagonal step followed by horizontal steps to return to the diagonal
with $\dia_\circ(P)=k$
and $\rise_\circ(P)=\ell$ are in bijection with the Dyck paths $P'$ with rise-touch composition
$(\beta,1)$, $\dia_\circ(P)=k$
and $\rise_\circ(P)=\ell-1$ by a cyclic rotation described at the beginning of this section.
We note that the $\circ$-area does not change with the cyclic rotation because the first row
is a $\circ$-decorated rise so $\area_\circ(P)=\area_\circ(P')$.  Since there is one diagonal
inversion of the form $(1,j)$ for each $j>1$ with $a_j(D)=0$,
then $\dinv_\circ(P) = \dinv_\circ(P')+\ell(\beta)$.

\vskip .1in
The three cases are disjoint, they cover all possible $\circ$-decorated Dyck paths with
rise-touch composition beginning with a $1$ and the weights on the left hand side of the equation agree
with those on the right hand side, hence equation \eqref{eq:firststepsmall} holds.
\end{proof}
\end{subsection}
\end{section}

\begin{section}{The symmetric function recurrence} \label{sec:sfrec}

In this section we will provide a proof of the following 
symmetric function identity that agrees with the combinatorial recurrence
on the generating function for $\circ$-decorated Schr\"oder paths.

\begin{theorem}\label{thm:mainrec} For $k\geq 0$, and for integers $d, r, m$,
\begin{align} \label{eq:mainrec}
\BC_{k+1}^\ast \Delta_{h_{d}}\! \nabla( e_r^\ast h_{m}^\ast )
= t^{k} \BB_{k}^\ast& \Delta_{h_d} \!\nabla( e_r^\ast h_{m-1}^\ast )
+ t^{k} \BB_{k+1}^\ast \Delta_{h_{d-1}}\! \nabla( e_r^\ast h_{m}^\ast )\\
&+ \chi(k=0) \Delta_{h_d} \!\nabla( e_{r-1}^\ast h_{m}^\ast )~.\label{eq:mainrec2}
\end{align}
\end{theorem}

Notice that at $d=0$, one of the terms is equal to $0$.
This case of this identity is equivalent to the recurrence used
in \cite{GXZ10} to prove the Schr\"oder case of the compositional shuffle conjecture.

We have as a consequence the following expression of coefficients which have combinatorial meaning.

\begin{cor} \label{cor:sfreccurence}For non-negative integers $k$, $\ell$ and $a$ and for a composition $\beta$,
\begin{align}
&\left< \Delta_{h_{\ell}} \nabla \BC_{a+1}(C_\beta), s_{k+1,1^{|\beta|+a-k}} \right> \label{eq:coeffrec1}
= t^{a}q^{\ell(\beta)} \sum_{\ga \models a} \left< \Delta_{h_\ell} \nabla C_{\beta,\ga}, s_{k+1,1^{|\beta|+a-k-1}} \right>\\
&+ t^{a}q^{\ell(\beta)} \sum_{\ga \models a+1}\left< \Delta_{h_{\ell-1}} 
\nabla C_{\beta,\ga}, s_{k+1,1^{|\beta|+a-k}}\right>
+ \chi(a=0)\left< \Delta_{h_\ell} \nabla (C_\beta), s_{k,1^{|\beta|-k}}\right>~. \label{eq:coeffrec2}
\end{align}
\end{cor}


\begin{proof}
This identity is derived by taking the $\ast$-scalar product with $C_\beta$ on both sides of
the equation \eqref{eq:mainrec}--\eqref{eq:mainrec2}.  We begin by taking the $\ast$-scalar product
of $C_\beta$ with the left hand side of \eqref{eq:mainrec}.  We note that since $\{\Ht_\mu\}_\mu$ is an
orthogonal basis with respect to the $\ast$-scalar product and are eigenvectors of the operators $\nabla$ and $\Delta_{h_\ell}$,
these operators are self dual with respect to the $\ast$-scalar product and commute with each other.
\begin{equation}
\left< C_\beta, \BC_{a+1}^\ast \Delta_{h_{\ell}}\! \nabla( e_r^\ast h_{n}^\ast ) \right>_\ast
= \left< \Delta_{h_{\ell}}\! \nabla \BC_{a+1}(C_\beta),  e_r^\ast h_{n}^\ast \right>_\ast =
\left< \Delta_{h_{\ell}}\! \nabla \BC_{a+1}(C_\beta),  h_r e_{n} \right>.
\end{equation}

Similarly, the scalar product with the other three terms
of the right hand side of \eqref{eq:mainrec}--\eqref{eq:mainrec2} have expressions involving only $C_\alpha$
once equation \eqref{eq:BCrelation} is applied to the expressions of the form $\BB_m(C_\beta)$.  We conclude that
\begin{align}
&\left< \Delta_{h_{\ell}} \nabla \BC_{a+1}(C_\beta), h_{k+1} e_{|\beta|+a-k} \right> 
= t^{a}q^{\ell(\beta)} \sum_{\ga \models a} \left< \Delta_{h_\ell} \nabla C_{\beta,\ga}, h_{k+1} e_{|\beta|+a-k-1} \right>\\
&+ t^{a}q^{\ell(\beta)} \sum_{\ga \models a+1}\left< \Delta_{h_{\ell-1}} 
\nabla C_{\beta,\ga}, h_{k+1} e_{|\beta|+a-k} \right>
+ \chi(a=0)\left< \Delta_{h_\ell} \nabla (C_\beta), h_k e_{|\beta|-k} \right>~.
\end{align}

Equations \eqref{eq:coeffrec1}--\eqref{eq:coeffrec2} then follow because 
$s_{k+1,1^{|\beta|+a-k}}=\sum_{r=0}^{|\beta|+a-k} (-1)^r h_{k+1+r} e_{|\beta|+a-k-r}$.
\end{proof}

Note that in the case that $k=0$ and $r=1$ that $s_{r-1,1^{n+1-r}}$ is $0$ if $n>0$ and $s_{r-1,1^{n+1-r}} = 1$
if $n=0$.

\begin{theorem} \label{thm:mainresult} For non-negative integers $k$ and $\ell$ and
a composition $\alpha$,
\begin{equation}\label{eq:induction}
\Catp_{\alpha,k,\ell}(q,t) = \left< \Delta_{h_\ell} \nabla( C_\alpha ), s_{k+1,1^{|\alpha|-k-1}} \right>~.
\end{equation}
\end{theorem}

\begin{proof}
We have just established in the previous section (combining Propositions \ref{prop:Catrec1} and \ref{prop:Catrec2})
that
\begin{align}\label{eq:firststepbig}
\Catp_{(a+1,\beta),k,\ell}(q,t)
= &t^{a} q^{\ell(\beta)} \sum_{\ga \models a} \Catp_{(\beta,\ga),k,\ell}(q,t) 
+ t^{a} q^{\ell(\beta)} \sum_{\ga \models a+1} \Catp_{(\beta,\ga),k,\ell-1}(q,t)\\
&+ \chi(a=0) \Catp_{\beta,k-1,\ell}(q,t)~. 
\end{align}
This combinatorial recurrence agrees with Corollary \ref{cor:sfreccurence} in the sense that if
$$\Catp_{(\beta,\ga),k,\ell}(q,t) = \left< \Delta_{h_\ell} \nabla C_{\beta,\ga}, s_{k+1,1^{|\beta|+a-k-1}} \right>$$
and 
$$\Catp_{(\beta,\ga),k,\ell-1}(q,t) = \left< \Delta_{h_{\ell-1}} 
\nabla C_{\beta,\ga}, s_{k+1,1^{|\beta|+a-k}}\right>$$ and
$$\Catp_{\beta,k-1,\ell}(q,t) = \left< \Delta_{h_\ell} \nabla (C_\beta), s_{k,1^{|\beta|-k}}\right>~,$$
then 
$$\Catp_{(a+1,\beta),k,\ell}(q,t) = \left< \Delta_{h_{\ell}} \nabla \BC_{a+1}(C_\beta), s_{k+1,1^{|\beta|+a-k}} \right>~.$$
The indices of the right hand side of this recurrence have the property that either the value of
$\ell$ is lower or the size of the composition $\alpha$ is smaller.
Therefore we proceed by induction by assuming that equation \eqref{eq:induction} holds true for
compositions of smaller size and smaller values of $\ell$.  Then it remains to show that
it is true for a base case.
 
We note that if $\alpha = (1)$, then
$$\left< \Delta_{h_\ell} \nabla( C_1 ), s_{k+1,1^{-k}} \right> = 1$$
if and only if $k=0$ and it is equal to $0$ otherwise.
Similarly, $\Catp_{(1),k,\ell}(q,t)=1$ if and only if $k=0$ (and $0$ otherwise)
because the generating function
for $\circ$-decorated Schr\"oder paths has one term for
the Schr\"oder path consisting of $\ell$ $\circ$-decorated rises,
a peak, followed by horizontal steps back to the diagonal.
\end{proof}

In order to prove our symmetric function identity from Theorem \ref{thm:mainrec} 
we break the calculation into lemmas that will hopefully make
a long calculation a little easier to follow.

\begin{lemma}\label{lem:multhbBmu} For integers $d$ and partitions $\mu$ and $\nu$,
\begin{align}\label{eq:lemmamulthbBmu}
\frac{T_\nu}{w_\nu} h_d[B_\mu] \Ht_\mu[D_\nu] =
\sum_{a \geq 0} (-1)^a  h_{d-a}\left[ \frac{1}{M}\right]
\sum_{\substack{\ga \supseteq \nu\\|\ga|=|\nu|+a}}
\Ht_\mu[D_\ga] \frac{T_\ga}{w_\ga} c_{\ga\nu}^{e_a^\perp}~.
\end{align}
\end{lemma}

\begin{proof}
First we apply equation \eqref{eq:alphaaddition} to $h_d[B_\mu]$ (the alphabet addition formula)
and show that
\begin{align}
h_d[B_\mu] = h_d\!\!\left[\frac{MB_\mu-1}{M} + \frac{1}{M} \right] = \sum_{a \geq 0} h_{a}\!\!\left[\frac{D_\mu}{M} \right]
h_{d-a}\!\!\left[\frac{1}{M} \right]~.
\end{align}
To the left hand side of equation \eqref{eq:lemmamulthbBmu} we apply the reciprocity formula \eqref{eq:GHTreciprocity}, 
and then use the generalized Pieri coefficients
that were introduced in \cite{GarHag02} from equation \eqref{eq:genpiericoeffs},
and then reapply the reciprocity formula to derive
\begin{align}
\frac{T_\nu}{w_\nu} h_d[B_\mu] \Ht_\mu[D_\nu] 
&= (-1)^{|\mu|+|\nu|} \frac{T_\mu}{w_\nu} \sum_{a \geq 0} h_{a}\!\!\left[\frac{D_\mu}{M} \right]
\Ht_\nu[D_\mu] h_{d-a}\!\!\left[\frac{1}{M} \right]\\
&=(-1)^{|\mu|+|\nu|} \frac{T_\mu}{w_\nu} \sum_{a \geq 0} \sum_{\substack{\ga \supseteq \nu\\|\ga|=|\nu|+a}}
d_{\ga\nu}^{\omega e_a^\ast} \Ht_\ga[D_\mu]
h_{d-a}\!\!\left[\frac{1}{M} \right]\\
&=\frac{T_\mu}{w_\nu} \sum_{a \geq 0} \sum_{\substack{\ga \supseteq \nu\\|\ga|=|\nu|+a}}
(-1)^a \frac{T_\ga}{T_\mu} \Ht_\mu[D_\ga]
h_{d-a}\!\!\left[\frac{1}{M} \right] d_{\ga\nu}^{\omega e_a^\ast}~.
\end{align}
Now recall that we can convert the $d_{\ga\nu}^{f}$ coefficients to $c_{\ga\nu}^{g\perp}$ coefficients using
equation \eqref{eq:genPiericoefrel}, the resulting equation is equal to the right hand side of the
equation stated in \eqref{eq:lemmamulthbBmu}.
\end{proof}

Now the coefficients $c_{\ga\nu}^{e_a\perp}$ have the sum over $\nu$ that was also calculated
by Garsia and Haglund \cite{GarHag02} and we need this expression that we state in the following lemma.

\begin{lemma}
For $a$ a non-negative integer and for a fixed partition $\gamma$,
\begin{align}
\sum_{\substack{\nu \subseteq \ga\\|\nu| = |\ga|- a}} c_{\ga\nu}^{e_a\perp} = e_a[B_\ga]
\end{align}
\end{lemma}
\begin{proof}
Of course if $|\ga| < a$ then both the left and right hand sides of this expression are equal to $0$.
From the identity in equation \eqref{eq:genPierisums} in the special case of $g = e_a$ we have
\begin{align}\label{eq:here1}
\sum_{\substack{\nu \subseteq \ga\\|\nu| = |\ga|- a}} c_{\ga\nu}^{e_a\perp} = 
\nabla^{-1} h_a\!\!\left[ \frac{X-\epsilon}{M} \right] \Big|_{X \rightarrow D_\ga}~.
\end{align}
We know that $\nabla^{-1} h_a\!\!\left[\frac{X}{M} \right] = e_a\!\!\left[\frac{X}{M} \right]$ and hence we apply the
alphabet addition formulas to simplify the right hand side of the equation \eqref{eq:here1} as
\begin{align}
\nabla^{-1} h_a\!\!\left[ \frac{X-\epsilon}{M} \right] \Big|_{X \rightarrow D_\ga}
&= \sum_{b \geq 0} \nabla^{-1} h_{a-b}\!\!\left[ \frac{X}{M} \right]
h_b\!\!\left[ \frac{-\epsilon}{M} \right]\Big|_{X \rightarrow D_\ga}\\
&= \sum_{b \geq 0}  e_{a-b}\!\!\left[ \frac{D_\ga}{M} \right]
e_b\!\!\left[ \frac{1}{M} \right]
= e_{a}\!\!\left[ \frac{M B_\ga -1}{M} + \frac{1}{M}\right] = e_a[B_\ga]~.\qedhere
\end{align}
\end{proof}

The following result gives us an expression for a kernel which we can use
to apply the $\BB$ and $\BC$ operators.
\begin{prop} \label{prop:kernel} For non-negative integers $d,r$ and $n$
\begin{align}
\Delta_{h_d} \nabla( e_r^\ast h_n^\ast ) = \sum_{a \geq 0} \sum_{\gamma}
(-1)^{r+a} e_{n+r}\!\!\left[\frac{X D_\ga}{M}\right] h_{d-a}\!\!\left[\frac{1}{M}\right]
h_{n+a-|\ga|}\!\!\left[\frac{-1}{M}\right]  \frac{T_\ga}{w_\ga} e_a[B_\ga]
\end{align}
where the sum over $\ga$ is over all partitions of size smaller than or equal to $n+a$.
\end{prop}

\begin{proof}  We begin by introducing an extra set of variables $W$ into the expression
and using the fact that
$e_r^\ast[X] h_n^\ast[X] = \left< e_{n+r}^\ast[XW], h_r[W] e_n[W] \right>$ where the scalar
product is taken with respect to symmetric functions in the variables $W$.
Then equation \eqref{eq:enexpansion} implies
\begin{align}
\Delta_{h_d} \nabla( e_r^\ast h_n^\ast ) &= \left< \Delta_{h_d} \nabla e_{n+r}^\ast[XW], h_r[W] e_n[W] \right>\\
&= \sum_{\mu \vdash n+r} \frac{ h_d[ B_\mu] T_\mu \Ht_\mu[X]}{w_\mu} \left< \Ht_\mu[W], h_r[W] e_n[W] \right>~.
\label{eq:here2}
\end{align}
Now a special case of the Macdonald coefficients that are known (see \cite{Mac95} Exercise 2 p. 362)
is the scalar product $\left< \Ht_\mu[W], h_r[W] e_n[W] \right> = e_n[B_\mu]$.  To apply
our Lemma \ref{lem:multhbBmu} we need an expression with $D_\mu = M B_\mu -1$, hence by the alphabet addition
formulae, we have

\begin{align}
&=\sum_{\mu \vdash n+r} \frac{ h_d[ B_\mu] T_\mu \Ht_\mu[X]}{w_\mu} e_n[B_\mu]\\
&=\sum_{\mu \vdash n+r} \frac{ h_d[ B_\mu] T_\mu \Ht_\mu[X]}{w_\mu} e_n^\ast[(MB_\mu-1) + 1]\\
&=\sum_{\mu \vdash n+r} \sum_{k\geq0} \frac{ h_d[ B_\mu] T_\mu \Ht_\mu[X]}{w_\mu} e_{k}^\ast[D_\mu] e_{n-k}^\ast[1]~.
\label{eq:here3}
\end{align}
We can then expand the expression $e_k^\ast[D_\mu] = \sum_{\nu \vdash k} \frac{\Ht_\nu[D_\mu]}{w_\nu}$ so
that we can apply the reciprocity formula from equation \eqref{eq:GHTreciprocity}.
\begin{align}
&=\sum_{\mu \vdash n+r} \sum_{k\geq0} \sum_{\nu \vdash k}
\frac{ h_d[ B_\mu] T_\mu \Ht_\mu[X]}{w_\mu} \frac{\Ht_\nu[D_\mu]}{w_\nu} e_{n-k}^\ast[1]\\
&=\sum_{\mu \vdash n+r} \sum_{k\geq0} \sum_{\nu \vdash k} (-1)^{n+r+k}
\frac{ h_d[ B_\mu] \Ht_\mu[X]}{w_\mu} \frac{T_\nu \Ht_\mu[D_\nu]}{ w_\nu} e_{n-k}^\ast[1]~.\label{eq:here4}
\end{align}
At this point we can apply equation \eqref{eq:lemmamulthbBmu} and simplify the
power of $-1$ by the expression $(-1)^{n+k} e_{n-k}^\ast[1] = h_{n-k}^\ast[-1]$.
We also combine the sum over $k \geq 0$ and $\nu \vdash k$ to just be a sum over all
partitions $\nu$.  The sum is actually finite because the expression $h_{n-k}^\ast[-1]$ is
equal to $0$ if $|\nu| > n$.  We then interchange the sum over $\nu$ and $\gamma$ then
we have the following manipulation of equation \eqref{eq:here4} to arrive at the expression
stated in the theorem.
\begin{align}
&=\sum_{\mu \vdash n+r} \sum_{\nu} (-1)^{r}
\frac{ \Ht_\mu[X]}{w_\mu}  h_{n-|\nu|}^\ast[-1]\sum_{a \geq 0} (-1)^a  h_{d-a}^\ast\left[ 1\right]
\sum_{\substack{\ga \supseteq \nu\\|\ga|=|\nu|+a}}
\Ht_\mu[D_\ga] \frac{T_\ga}{w_\ga} c_{\ga\nu}^{e_a\perp}\\
&=\sum_{\mu \vdash n+r} \sum_{\ga}
 \sum_{a \geq 0} (-1)^{r+a}
\frac{ \Ht_\mu[X] \Ht_\mu[D_\ga]}{w_\mu}  h_{n+a-|\ga|}^\ast[-1]  h_{d-a}^\ast\left[ 1\right]
 \frac{T_\ga}{w_\ga} \sum_{\substack{\nu \subseteq \ga\\|\nu|=|\ga|-a}} c_{\ga\nu}^{e_a\perp}\\
&= \sum_{\ga}
 \sum_{a \geq 0} (-1)^{r+a}
e_{n+r}^\ast[XD_\ga]  h_{n+a-|\ga|}^\ast[-1]  h_{d-a}^\ast\left[ 1\right]
 \frac{T_\ga}{w_\ga} e_a[B_\ga]~.\qedhere
\end{align}
\end{proof}

Now we are able to apply the $\BC$ and $\BB$ operators to prove Theorem \ref{thm:mainrec}.

To simplify our calculation, it will help to have expressions for the action of
$\BC_m^\ast$ on $e_c\!\!\left[ \frac{X D_\ga}{M} \right]$.  

\begin{lemma}  For a non-negative integer $c$ and integer $m$,
\begin{equation} \label{eq:BCesc}
\BC_m^\ast e_c\!\!\left[ \frac{X D_\ga}{M} \right] = 
\sum_{b \geq 0}
q^{-b+1} (-1)^{m+1}
e_{c-b}\!\!\left[\frac{XD_\ga}{M}\right] h_b[D_\ga]
e_{b-m}\!\!\left[\frac{X}{1-t} \right]~.
\end{equation}
\end{lemma}

\begin{proof}
\begin{align}
\BC_m^\ast e_c\!\!\left[ \frac{X D_\ga}{M} \right] &= (-q)
e_c\!\!\left[\left(X-\frac{ M}{qu}\right)\frac{D_\ga}{M}  \right]
\Omega\!\!\left[\frac{- u X}{1-t} \right]
\Big|_{u^{-m}}\\
&=\sum_{a \geq 0} \sum_{b \geq 0} u^{a-b}
q^{-b+1} (-1)^{a+b+1}
e_{c-b}\!\!\left[\frac{XD_\ga}{M}\right] h_b\left[D_\ga\right]
e_a\!\!\left[\frac{X}{1-t} \right]
\Big|_{u^{-m}}
\end{align}
so when we take the coefficient of $u^{-m}$ in $u^{a-b}$, then
$-m = a-b$ or $a=b-m$ and we obtain the expression stated in equation \eqref{eq:BCesc}.
\end{proof}

We also develop a
full expression for $\BB_m^\ast$ on the kernel from Proposition \ref{prop:kernel}
because to prove the theorem we will expand the left hand side of \eqref{eq:mainrec}
and need to recognize when we have the right hand side of that equation.

\begin{lemma} \label{lem:bmskern} For all integers $m, r$ and $a$ and for non negative integer $d$,
\begin{align}
\BB_{m}^\ast \Delta_{h_d} \nabla( e_r^\ast h_n^\ast ) &= \sum_{a \geq 0} \sum_{b \geq 0} \sum_{\tau}
(-1)^{r+a+b+m} e_b[D_\tau] e_{n+r-b}\!\!\left[\frac{X D_\tau}{M} \right]  \times\\
&\hskip .6in \times
e_{b-m}\!\!\left[\frac{X}{1-t} \right] h_{d-a}^\ast[1]
h_{n+a-|\tau|}^\ast[-1]  \frac{T_\tau}{w_\tau} e_a[B_\tau]~.
\end{align}
\end{lemma}

\begin{proof}
As we did in the previous lemma, we will first develop an expression for the action
of $\BB_m^\ast$ on $e_{n+r}^\ast[X D_\ga]$.
\begin{align}
\BB_m^\ast e_{n+r}^\ast[X D_\tau]
&= e_{n+r}\!\!\left[\left(X+\frac{M}{u}\right) \frac{D_\tau}{M} \right]\Omega\!\!\left[\frac{-uX}{1-t} \right]
\Big|_{u^{-m}}\\
&=\sum_{a \geq 0} \sum_{b \geq 0}
u^{a-b} e_b[D_\tau] e_{n+r-b}\!\!\left[\frac{X D_\tau}{M} \right]  h_a\!\!\left[\frac{-X}{1-t} \right]\Big|_{u^{-m}}~.
\end{align}
Now when we take the coefficient of $u^{-m}$ in this expression $-m = a-b$, hence $a = b-m$
and our expression becomes
\begin{align}\label{eq:BBonecXDtauoM}
\BB_m^\ast e_{n+r}^\ast[X D_\tau]&=\sum_{b \geq 0} (-1)^{b+m}
e_b[D_\tau] e_{n+r-b}\!\!\left[\frac{X D_\tau}{M} \right]  e_{b-m}\!\!\left[\frac{X}{1-t} \right]~.
\end{align}
Now we apply that expression to the kernel that we derived in Proposition \ref{prop:kernel}.
\begin{align}
\BB_m^\ast &\Delta_{h_d} \nabla( e_r^\ast h_n^\ast )\\ &= \sum_{a \geq 0} \sum_{\tau}
(-1)^{r+a} \BB_m^\ast e_{n+r}\!\!\left[\frac{X D_\tau}{M}\right] h_{d-a}^\ast[1]
h_{n+a-|\tau|}^\ast[-1]  \frac{T_\tau}{w_\tau} e_a[B_\tau]\\
&= \sum_{a \geq 0} \sum_{b \geq 0} \sum_{\tau}
(-1)^{r+a+b+m} e_{n+r-b}\!\!\left[\frac{X D_\tau}{M} \right] e_b[D_\tau] \times\\
&\hskip .6in \times
e_{b-m}\!\!\left[\frac{X}{1-t} \right] h_{d-a}^\ast[1]
h_{n+a-|\tau|}^\ast[-1]  \frac{T_\tau}{w_\tau} e_a[B_\tau]~.\qedhere
\end{align}
\end{proof}

We are now prepared to prove Theorem \ref{thm:mainrec} by a direct calculation using the results we
have derived above.

\begin{proof} (of Theorem \ref{thm:mainrec})
We begin by applying the operator $\BC_{k+1}$ to the expression in Proposition \ref{prop:kernel}.
Equation \eqref{eq:BCesc} with $m \rightarrow k+1$, $c \rightarrow n+r$ says that it is equal to
\begin{align}
\BC_{k+1}^\ast \Delta_{h_d} \nabla( e_r^\ast h_n^\ast ) &= \sum_{a \geq 0} \sum_{\gamma}\sum_{b \geq 0}
q^{-b+1} (-1)^{r+a+k}
e_{n+r-b}^\ast[XD_\ga] h_b[D_\ga] \times \\
& \hskip .6in \times e_{b-k-1}\!\!\left[\frac{X}{1-t} \right]
h_{d-a}^\ast[1]
h_{n+a-|\ga|}^\ast[-1] \frac{T_\ga}{w_\ga} e_a[B_\ga]~.\label{eq:here5}
\end{align}
Now we know that since $k\geq 0$, then $b \geq 1$ since $e_{b-k-1} = 0$ for $b=0$.  In fact, it may be helpful to make the replacement $b \rightarrow b+1$ in \eqref{eq:here5} and we can apply equation \eqref{eq:hkDga} to $h_{b+1}[D_\ga]$.
\begin{align}
&= \sum_{a \geq 0} \sum_{\gamma}\sum_{b \geq 0}
q^{-b} (-1)^{r+a+k}
e_{n+r-b-1}^\ast[XD_\ga] h_{b+1}[D_\ga] \times \\
& \hskip .6in \times e_{b-k}\!\!\left[\frac{X}{1-t} \right]
h_{d-a}^\ast[1]
h_{n+a-|\ga|}^\ast[-1] \frac{T_\ga}{w_\ga} e_a[B_\ga]\\
&=
\sum_{a \geq 0} \sum_{\gamma}\sum_{b \geq 0} \sum_{\tau \rightarrow \ga}
 (-1)^{r+a+k}M t^{b}  c_{\ga\tau} \left( \frac{T_\ga}{T_\tau} \right)^{b}
e_{n+r-b-1}^\ast[XD_\ga] \times \label{eq:here6a}\\
& \hskip .6in \times e_{b-k}\!\!\left[\frac{X}{1-t} \right]
h_{d-a}^\ast[1]
h_{n+a-|\ga|}^\ast[-1] \frac{T_\ga}{w_\ga} e_a[B_\ga]\label{eq:here6b}\\
&\hskip .2in-\chi(k=0)\sum_{a \geq 0} \sum_{\gamma}
(-1)^{r+a} e_{n+r-1}^\ast[XD_\ga] 
h_{d-a}^\ast[1]
h_{n+a-|\ga|}^\ast[-1] \frac{T_\ga}{w_\ga} e_a[B_\ga]~.\label{eq:here6}
\end{align}
Notice already that equation \eqref{eq:here6} is equal to $+\chi(k=0) \Delta_{h_d} \nabla( e_{r-1}^\ast h_n^\ast)$.
Then it remains to expand equations \eqref{eq:here6a}--\eqref{eq:here6b}.  For this we use
$B_\ga = B_\tau + \frac{T_\ga}{T_\tau}$ and $D_\ga = D_\tau + M \frac{T_\ga}{T_\tau}$.
In this case $e_a[B_\ga] = e_a[B_\tau] +
e_{a-1}[B_\tau]\frac{T_\ga}{T_\tau}$ and 
$e_n^\ast[XD_\ga] = \sum_{c \geq 0} e_{n-c}^\ast[XD_\tau] e_c\!\!\left[X\frac{T_\ga}{T_\tau}\right]$.
We will also use
the identity $t^b e_{b-k}\!\!\left[\frac{X}{1-t} \right] = t^k e_{b-k}\!\!\left[\frac{tX}{1-t} \right]$
to show that \eqref{eq:here6a}--\eqref{eq:here6b} is equivalent to the following expression:
\begin{align}
&=
t^k \sum_{a \geq 0} \sum_{\gamma}\sum_{b \geq 0} \sum_{\tau \rightarrow \ga}
\sum_{c \geq 0}
 (-1)^{r+a+k}M c_{\ga\tau} \left( \frac{T_\ga}{T_\tau} \right)^{b}
e_{n+r-b-1-c}^\ast[XD_\tau] e_c\!\!\left[X\frac{T_\ga}{T_\tau}\right] \times \label{eq:here7a}\\
& \hskip .6in \times e_{b-k}\!\!\left[\frac{tX}{1-t} \right]
h_{d-a}^\ast[1]
h_{n+a-|\ga|}^\ast[-1] \frac{T_\ga}{w_\ga} \left(e_{a}[B_\tau] + e_{a-1}[B_\tau]\frac{T_\ga}{T_\tau}\right)~.\label{eq:here7b}
\end{align}
We will regroup the $T_\ga/T_\tau$ terms and break the expression into two separate sums.  We will also
replace $c_{\ga\tau}$ with $d_{\ga\tau}$ using equation \eqref{eq:cmunudmunurel}.
Then equations \eqref{eq:here7a}--\eqref{eq:here7b} are equivalent to
\begin{align}
&=
t^k \sum_{a \geq 0} \sum_{\gamma}\sum_{b \geq 0} \sum_{\tau \rightarrow \ga}
\sum_{c \geq 0}
 (-1)^{r+a+k}d_{\ga\tau} \left( \frac{T_\ga}{T_\tau} \right)^{b+c+1}
e_{n+r-b-1-c}^\ast[XD_\tau] e_c[X] \times \label{eq:here8a}\\
& \hskip .6in \times e_{b-k}\!\!\left[\frac{tX}{1-t} \right]
h_{d-a}^\ast[1]
h_{n+a-|\ga|}^\ast[-1] \frac{T_\ga}{w_\tau} e_{a}[B_\tau]~.\label{eq:here8b}\\
&\hskip .2in + t^k \sum_{a \geq 0} \sum_{\gamma}\sum_{b \geq 0} \sum_{\tau \rightarrow \ga}
\sum_{c \geq 0}
 (-1)^{r+a+k}d_{\ga\tau} \left( \frac{T_\ga}{T_\tau} \right)^{b+c+2}
e_{n+r-b-1-c}^\ast[XD_\tau] e_c[X] \times \label{eq:here8c}\\
& \hskip .6in \times e_{b-k}\!\!\left[\frac{tX}{1-t} \right]
h_{d-a}^\ast[1]
h_{n+a-|\ga|}^\ast[-1] \frac{T_\tau}{w_\tau} e_{a-1}[B_\tau]~.\label{eq:here8d}
\end{align}
In both of these sums, we can interchange the sum over partitions $\gamma$ and then over $\tau \rightarrow \ga$
to a sum over partitions $\tau$ and then over $\ga \leftarrow \tau$.  In this case the sums
of the form $\sum_{\ga \leftarrow \tau} d_{\ga\tau}\left( \frac{T_\ga}{T_\tau} \right)^{n} = (-1)^{n-1} e_{n-1}[D_\tau]$
(where $n \in \{ b+c+1, b+c+2\}$)
by equation \eqref{eq:esDga} since in both of the expressions we have $n>1$ so there is no $\chi(n=0)$ term.
Equations \eqref{eq:here8a}--\eqref{eq:here8d}
are equivalent to
\begin{align}
&=
t^k \sum_{a \geq 0} \sum_{b \geq 0} 
\sum_{c \geq 0} \sum_{\tau}
 (-1)^{r+a+k+b+c}e_{b+c}[D_\tau]
e_{n+r-b-1-c}^\ast[XD_\tau] e_c[X] \times \label{eq:here9a}\\
& \hskip .6in \times e_{b-k}\!\!\left[\frac{tX}{1-t} \right]
h_{d-a}^\ast[1]
h_{n+a-|\tau|-1}^\ast[-1] \frac{T_\tau}{w_\tau} e_{a}[B_\tau]~.\label{eq:here9b}\\
&\hskip .2in + t^k \sum_{a \geq 0} \sum_{b \geq 0} \sum_{c \geq 0} \sum_{\tau}
 (-1)^{r+a+k+b+c+1}e_{b+c+1}[D_\tau]
 e_{n+r-b-1-c}^\ast[XD_\tau] e_c[X] \times \label{eq:here9c}\\
& \hskip .6in \times e_{b-k}\!\!\left[\frac{tX}{1-t} \right]
h_{d-a}^\ast[1]
h_{n+a-|\tau|-1}^\ast[-1] \frac{T_\tau}{w_\tau} e_{a-1}[B_\tau]~.\label{eq:here9d}
\end{align}
Instead of summing over $b\geq0$ and $c\geq 0$, we will let $v=b+c$ and sum over $v \geq 0$ and
$0 \leq b \leq v$ and let $c = v-b$.  These replacements make
equations \eqref{eq:here9a}--\eqref{eq:here9d} equivalent to
\begin{align}
&=
t^k \sum_{a \geq 0} \sum_{v \geq 0} 
\sum_{b=0}^v \sum_{\tau}
 (-1)^{r+a+k+v}e_{v}[D_\tau]
e_{n+r-v-1}^\ast[XD_\tau] e_{v-b}[X] \times \label{eq:here10a}\\
& \hskip .6in \times e_{b-k}\!\!\left[\frac{tX}{1-t} \right]
h_{d-a}^\ast[1]
h_{n+a-|\tau|-1}^\ast[-1] \frac{T_\tau}{w_\tau} e_{a}[B_\tau]~.\label{eq:here10b}\\
&\hskip .2in + t^k \sum_{a \geq 0} \sum_{v \geq 0} \sum_{b=0}^v \sum_{\tau}
 (-1)^{r+a+k+v+1}e_{v+1}[D_\tau]
 e_{n+r-v-1}^\ast[XD_\tau] e_{v-b}[X] \times \label{eq:here10c}\\
& \hskip .6in \times e_{b-k}\!\!\left[\frac{tX}{1-t} \right]
h_{d-a}^\ast[1]
h_{n+a-|\tau|-1}^\ast[-1] \frac{T_\tau}{w_\tau} e_{a-1}[B_\tau]~.\label{eq:here10d}
\end{align}

Notice now that the sub expression
\begin{equation}
\sum_{b=0}^v e_{v-b}[X] e_{b-k}\!\!\left[\frac{tX}{1-t} \right]
= e_{v-k}\!\!\left[X + \frac{tX}{1-t} \right]
= e_{v-k}\!\!\left[\frac{X}{1-t} \right]~.
\end{equation}
%
By replacing this in the sum, equations \eqref{eq:here10a}--\eqref{eq:here10d} are equivalent to
\begin{align}
&=
t^k \sum_{a \geq 0} \sum_{v \geq 0} \sum_{\tau}
 (-1)^{r+a+k+v}e_{v}[D_\tau]
e_{n+r-v-1}^\ast[XD_\tau] \times \label{eq:here11a}\\
& \hskip .6in \times e_{v-k}\!\!\left[\frac{X}{1-t} \right]
h_{d-a}^\ast[1]
h_{n+a-|\tau|-1}^\ast[-1] \frac{T_\tau}{w_\tau} e_{a}[B_\tau]~.\label{eq:here11b}\\
&\hskip .2in + t^k \sum_{a \geq 0} \sum_{v \geq 0} \sum_{\tau}
 (-1)^{r+a+k+v+1}e_{v+1}[D_\tau]
 e_{n+r-v-1}^\ast[XD_\tau] \times \label{eq:here11c}\\
& \hskip .6in \times e_{v-k}\!\!\left[\frac{X}{1-t} \right]
h_{d-a}^\ast[1]
h_{n+a-|\tau|-1}^\ast[-1] \frac{T_\tau}{w_\tau} e_{a-1}[B_\tau]~.\label{eq:here11d}
\end{align}
We then note that equation \eqref{eq:here11a}--\eqref{eq:here11b} is the right hand
side of the expression derived in Lemma \ref{lem:bmskern} with $m \rightarrow k$ and
$n \rightarrow n-1$ and hence is equal to $t^k \BB_k^\ast \Delta_{h_{d}} \nabla( e_{r}^\ast h_{n-1}^\ast )$.
As well we have that \eqref{eq:here11c}--\eqref{eq:here11d}
with $v \rightarrow v-1$ and $a \rightarrow a+1$ is equivalent to the expression
\begin{align}
&t^k \sum_{a \geq -1} \sum_{v \geq 1} \sum_{\tau}
 (-1)^{r+a+k+v+1}e_{v}[D_\tau]
 e_{n+r-v}^\ast[XD_\tau] \times \label{eq:here12c}\\
& \hskip .6in \times e_{v-k-1}\!\!\left[\frac{X}{1-t} \right]
h_{d-a-1}^\ast[1]
h_{n+a-|\tau|}^\ast[-1] \frac{T_\tau}{w_\tau} e_{a}[B_\tau]\label{eq:here12d}
\end{align}
and with a few exceptions of terms which are equal to $0$ (i.e. $a=-1$ and $v=0$) this
is precisely the right hand side of the expression in Lemma \ref{lem:bmskern}
with $m \rightarrow k+1$, $d \rightarrow d-1$ and hence is
equal to $t^k \BB_{k+1}^\ast\Delta_{h_{d-1}} \nabla( e_{r}^\ast h_{n}^\ast)$.
\end{proof}
\end{section}

\begin{section}{Remarks}
Section 7 of \cite{HRW15} has two conjectures that we do not resolve here.
These techniques may potentially be adapted to prove them as well, but
some additional work remains.

The first conjecture which does not follow directly from Theorem
\ref{thm:mainresult} is a symmetry property that is part of
Conjecture 7.1 in \cite{HRW15}.
\begin{conj}  For non-negative integers $k, \ell$ and $n > k+\ell$,
\begin{equation}
\left< \Delta_{h_\ell} \nabla e_{n - \ell}, s_{k+1,1^{n-\ell-k-1}} \right> = 
\left< \Delta_{h_k} \nabla e_{n - k}, s_{\ell+1,1^{n-\ell-k-1}} \right>~.
\end{equation}
\end{conj}
A proof of this result was announced recently by Michele D'Adderio and Anna Vanden Wyngaerd
\cite{AW17}.

A second conjecture is a combinatorial interpretation in terms of a second
type of decorated Dyck paths where vertical edges
except for the bottom most left one can be decorated.
This second combinatorial interpretation does
not seem to be compatible with the
compositional refinement as it is currently formulated.
However it is conjectured that it is compatible with
coefficients of the form $\left< \Delta_{h_k} \nabla E_{n-k,p}, s_{\ell+1, 1^{n-k-\ell-1}}
\right>$ (see Conjecture 7.2 of \cite{HRW15}) where 
$$E_{n-k,r} = \sum_{\substack{\alpha \models n-k\\\ell(\alpha)=r}} C_\alpha~.$$
Potentially recurrences on these coefficients 
that are compatible with the interpretation with this other type of decorated Dyck path
can be derived from Corollary \ref{cor:sfreccurence}.

The Delta Conjecture is a combinatorial interpretation for the
coefficients $\langle \Delta_{e_\ell}e_n, h_\lambda \rangle$ for $\lambda \vdash n$.
We would guess that a
different approach than was developed for the
Shuffle Conjecture may be needed because
the coefficients $\langle \Delta_{e_\ell} C_\alpha, h_\lambda \rangle$
are not even polynomials in $q$ and $t$ unless $\ell = |\alpha|$ (and this
case is the Compositional Shuffle Conjecture).

It seems possible that one of the combinatorial interpretations for the $q,t,w,z$-Catalan
might extend to give an interpretation for
$\langle \Delta_{h_\ell} \nabla C_\alpha, h_\lambda \rangle$.
In this case, the techniques developed by
Carlsson and Mellit \cite{CM15} might prove to be useful 
since those coefficients are compatible with the compositional refinement.
The extensions ideally are leading us in the direction of explaining
the following conjecture.

\begin{conj} For non-negative integers $\ell$ and $n > k+\ell$ and a composition
$\alpha \models n-\ell$ and a partition $\lambda \vdash n-\ell$,
\begin{equation}
\left< \Delta_{s_\mu} \nabla C_\alpha, s_\lambda \right>
\end{equation}
are polynomials in $q$ and $t$ with non-negative integer coefficients.
\end{conj}

While there has been significant recent progress on the monomial expansions
of $\nabla(f)$ for various symmetric functions $f$, passing to the Schur
expansions is an important major hurdle.

\end{section}

\end{document}